\documentclass[12pt,british,reqno]{article}

\usepackage{babel}

\usepackage[T1]{fontenc}

\usepackage[font=small,format=plain,labelfont=bf,up,textfont=it,up]{caption}
\usepackage{afterpage}
\usepackage{epsfig,afterpage}
\usepackage{graphicx}
\usepackage{wrapfig}

\DeclareFontFamily{T1}{calligra}{}
\DeclareFontShape{T1}{calligra}{m}{n}{<->s*[1.44]callig15}{}
\DeclareMathAlphabet\mathrsfso      {U}{rsfso}{m}{n}
\usepackage{amsthm,amssymb,amsmath} 
\usepackage[latin1]{inputenc}
\usepackage{enumerate}
\usepackage[all]{xy}

\usepackage{multicol}

\usepackage{mathtools}

\usepackage{tikz}
\usetikzlibrary{calc}

\makeatletter
\def\@map#1#2[#3]{\mbox{$#1 \colon\thinspace #2 \longrightarrow #3$}}
\def\map#1#2{\@ifnextchar [{\@map{#1}{#2}}{\@map{#1}{#2}[#2]}}
\g@addto@macro\@floatboxreset\centering
\makeatother

\renewcommand{\epsilon}{\ensuremath{\varepsilon}}
\renewcommand{\phi}{\ensuremath{\varphi}}

\renewcommand{\to}{\ensuremath{\longrightarrow}}
\renewcommand{\mapsto}{\ensuremath{\longmapsto}}

\newcommand{\Z}{\ensuremath{\mathbb{Z}}}

\newcommand{\ang}[1]{\ensuremath{\left\langle #1\right\rangle}}

\newtheoremstyle{theoremm}{}{}{\itshape}{}{\scshape}{.}{ }{}
\theoremstyle{theoremm}
\newtheorem{thm}{Theorem}

\newtheorem{prop}[thm]{Proposition}
\newtheorem{cor}[thm]{Corollary}
\newtheoremstyle{remark}{}{}{}{}{\scshape}{.}{ }{}

\theoremstyle{remark}
\newtheorem{defn}[thm]{Definition}
\newtheorem{rem}[thm]{Remark}

\newtheorem{que}[thm]{Question}

\renewcommand{\ker}[1]{\ensuremath{\operatorname{\text{Ker}}\left({#1}\right)}}

\newcommand{\rethm}[1]{Theorem~\protect\ref{thm:#1}}

\newcommand{\reprop}[1]{Proposition~\protect\ref{prop:#1}}
\newcommand{\recor}[1]{Corollary~\protect\ref{cor:#1}}

\numberwithin{equation}{section}
\allowdisplaybreaks

\usepackage[
bookmarks=true,
breaklinks=true,
bookmarksnumbered = true,
colorlinks= true,
urlcolor= green,
anchorcolor = yellow,
citecolor=blue,
]{hyperref}                   

\begin{document}

\title{On virtual singular braid groups}

\author{OSCAR~OCAMPO~\\
Departamento de Matem\'atica - Instituto de Matem\'atica e Estat\'istica,\\
Universidade Federal da Bahia,\\
CEP:~40170-110 - Salvador - Ba - Brazil.\\
e-mail:~\url{oscaro@ufba.br}
}

\date{}

\maketitle

\begin{abstract}

The virtual singular braid group arises as a natural common generalization of classical singular braid groups and virtual braid groups. 
In this paper, we study several algebraic properties of the virtual singular braid group $VSG_n$. 
We introduce numerical invariants for virtual singular braids arising from exponent sums of words in $VSG_n$, and describe explicitly the kernels of the associated homomorphisms onto abelian groups. 
We then determine all group homomorphisms, up to conjugation, from $VSG_n$ to the symmetric group $S_n$, and obtain corresponding semi-direct product decompositions. 
In the particular case $n=2$, we provide explicit presentations and algebraic descriptions of the kernels. 
Moreover, we show that certain relations are forbidden in $VSG_n$, and we introduce and study natural quotients of the virtual singular braid group, including welded and unrestricted versions, for which analogous structural results are obtained.

\end{abstract}

\let\thefootnote\relax

\section{Introduction}

There exist several generalizations of the Artin braid group $B_n$, both from geometric and algebraic points of view, and their study constitutes an active line of research. 
Recently, Caprau, Pena and McGahan~\cite{CPM} introduced virtual singular braids as a common generalization of classical singular braids, defined by Birman~\cite{Bi} and Baez~\cite{Ba} in the study of Vassiliev invariants, and virtual braids, defined by Kauffman~\cite{Kau} and Vershinin~\cite{V}.
In~\cite{CPM}, the authors proved Alexander and Markov type theorems for virtual singular braids and gave two presentations for the monoid of virtual singular braids, denoted by $VSB_n$.

In a subsequent paper, Caprau and Yeung~\cite{CY} showed that the monoid $VSB_n$ embeds into a group, called the \emph{virtual singular braid group} on $n$ strands and denoted by $VSG_n$.
They also provided a presentation for the virtual singular pure braid group $VSPG_n$ and proved that $VSG_n$ decomposes as a semi-direct product of $VSPG_n$ and the symmetric group $S_n$.

The interest in these objects has been growing, and some progress in their study has been made in recent years.
For instance, Caprau and Zepeda~\cite{CZ} constructed representations of the monoid $VSB_n$ and, using the Reidemeister--Schreier algorithm, obtained a presentation for the virtual singular pure braid monoid.
Moreover, Cisneros de la Cruz and Gandolfi~\cite{CG} studied algebraic, combinatorial, and topological properties of singular virtual braid monoids.

In this paper, we study several properties of the virtual singular braid group $VSG_n$, as well as some of its subgroups and quotients.
In \rethm{properties}, we establish algebraic properties of $VSG_n$, showing, for instance, that it is not residually nilpotent for $n \ge 3$ and that its commutator subgroup is perfect for $n \ge 5$, as in the classical, virtual, and singular braid group cases (see~\cite{BB,DG}).

We introduce numerical invariants for virtual singular braids arising from exponent sums of words in $VSG_n$ and describe in \rethm{goldbsingular} the kernels of the associated homomorphisms.
In particular, one of these kernels coincides with the normal closure of the virtual braid group inside $VSG_n$.
We then determine all group homomorphisms, up to conjugation, from the virtual singular braid group $VSG_n$ to the symmetric group $S_n$; see \reprop{homphi} and \rethm{BP}.
This problem was previously studied for virtual braid groups in~\cite{BP}.

Moreover, in \reprop{semi-direct}, we describe $VSG_n$ as a semi-direct product of the kernel of each admissible homomorphism and the symmetric group.
In the particular case $n = 2$, we study in detail the kernels of the homomorphisms $VSG_2 \to S_2$.
More precisely, in \rethm{kernels} and \recor{kernels}, we provide explicit presentations and algebraic descriptions for the kernel in each case.

We also show in \rethm{forbidden} that certain relations are forbidden in $VSG_n$.
Motivated by analogous constructions in the virtual braid group setting (see~\cite{B,BBD,D,Kau,KL}),
we introduce several quotients of the virtual singular braid group, including the welded singular braid group and the unrestricted virtual singular braid group, among other related groups.
Finally, for the quotients of $VSG_n$ considered in this paper, we establish results analogous to those obtained for $VSG_n$ itself.

Finally, we mention a convention that will be used throughout the paper. 
Let $G$ be a group and let $N$ be a normal subgroup of $G$. 
By a slight abuse of notation, we sometimes use the same symbol to denote an element of $G$ and its corresponding equivalence class in the quotient group $G/N$.

\subsection*{Acknowledgments}

The author gratefully acknowledges the support received from Eliane Santos, the staff of HCA, Bruno Noronha, Luciano Macedo, M\'arcio Isabella, Andreia de Oliveira Rocha, Andreia Gracielle Santana, Ednice de Souza Santos, and Vinicius Aiala, as well as from the SMURB--UFBA (Servi\c{c}o M\'edico Universit\'ario Rubens Brasil Soares), whose assistance since July~2024 was essential for the completion of this work.

The author was partially supported by the National Council for Scientific and Technological Development (CNPq, Brazil) through a \textit{Bolsa de Produtividade} grant~305422/2022--7. 
The author is grateful to the anonymous referee for a thorough and careful reading of the manuscript and for constructive comments and suggestions that substantially improved its clarity and presentation.

\section{The virtual singular braid group}\label{vsgn}

In this section, we introduce the definitions and presentations of the main groups considered in this paper.
We then establish several properties of the virtual singular braid group, introduce numerical invariants, and describe all possible group homomorphisms, up to conjugation, from the virtual singular braid group $VSG_n$ to the symmetric group $S_n$. 
The case $n = 2$ is treated in detail. 
For each admissible homomorphism, we obtain a decomposition of $VSG_n$ as a semi-direct product of the kernel of the homomorphism and the symmetric group.

\subsection{Definitions}

Let $n \ge 2$ be a positive integer.
The \emph{virtual singular braid monoid} on $n$ strands, denoted by $VSB_n$, is the monoid generated by the elements
\[
\sigma_1^{\pm 1},\dots,\sigma_{n-1}^{\pm 1},\;
\tau_1,\dots,\tau_{n-1},\;
v_1,\dots,v_{n-1},
\]
subject to the relations listed below.

The generators $\sigma_i$ correspond to classical crossings, the generators $\tau_i$ correspond to singular crossings,
and the generators $v_i$ correspond to virtual crossings. 
Geometrically, these generators are represented as in Figure~\ref{fig:generators}.

\begin{figure}[h]
\centering
\begin{tikzpicture}[scale=1, line cap=round, line join=round]

\def\h{1.2}     
\def\gap{4.4}   

\def\xa{-0.40}
\def\xb{ 0.40}

\def\xLone{-1.35}
\def\xLtwo{-0.85}
\def\xRone{ 0.85}
\def\xRtwo{ 1.35}

\newcommand{\passives}{%
  \draw[line width=0.8pt] (\xLone,0) -- (\xLone,\h);
  \draw[line width=0.8pt] (\xLtwo,0) -- (\xLtwo,\h);
  \node at ({(\xLone+\xLtwo)/2},0.60) {$\cdots$};

  \draw[line width=0.8pt] (\xRone,0) -- (\xRone,\h);
  \draw[line width=0.8pt] (\xRtwo,0) -- (\xRtwo,\h);
  \node at ({(\xRone+\xRtwo)/2},0.60) {$\cdots$};
}

\begin{scope}[shift={(0,0)}]
  \node at (0,\h+0.35) {$\sigma_i$};

  \passives

  \draw[line width=0.9pt] (\xb,0) .. controls (\xb,0.55) and (\xa,0.65) .. (\xa,\h);

  \draw[line width=0.9pt] (\xa,0) .. controls (\xa,0.55) and (\xb,0.65) .. (\xb,\h);

  \draw[line width=4pt, white] (-0.07,0.62) -- (0.07,0.58);
  \draw[line width=0.9pt] (\xa,0) .. controls (\xa,0.55) and (\xb,0.65) .. (\xb,\h);
\end{scope}

\begin{scope}[shift={(\gap,0)}]
  \node at (0,\h+0.35) {$\tau_i$};

  \passives

  \draw[line width=0.9pt] (\xa,0) .. controls (\xa,0.55) and (\xb,0.65) .. (\xb,\h);
  \draw[line width=0.9pt] (\xb,0) .. controls (\xb,0.55) and (\xa,0.65) .. (\xa,\h);

  \fill (0,0.60) circle (1.3pt);
\end{scope}

\begin{scope}[shift={(2*\gap,0)}]
  \node at (0,\h+0.35) {$v_i$};

  \passives

  \draw[line width=0.9pt] (\xa,0) .. controls (\xa,0.55) and (\xb,0.65) .. (\xb,\h);
  \draw[line width=0.9pt] (\xb,0) .. controls (\xb,0.55) and (\xa,0.65) .. (\xa,\h);

  \draw[line width=0.9pt] (0,0.60) circle (0.16);
\end{scope}

\end{tikzpicture}
\caption{The generators of the virtual singular braid group: classical crossings $\sigma_i$, singular crossings $\tau_i$, and virtual crossings $v_i$, acting on strands $i$ and $i+1$.}
\label{fig:generators}
\end{figure}

The monoid $VSB_n$ forms a well-defined algebraic structure under concatenation of braids.
Following~\cite{CPM}, this monoid embeds into a group, called the \emph{virtual singular braid group} on $n$ strands and denoted by $VSG_n$.

\begin{defn}\label{defn:vsg}
The \emph{virtual singular braid group} $VSG_n$ is the group generated by
\[
\sigma_1,\dots,\sigma_{n-1},\;
\tau_1,\dots,\tau_{n-1},\;
v_1,\dots,v_{n-1},
\]
subject to the following relations, for all admissible indices.
\medskip

\noindent\textbf{(1) Classical braid relations:}
\begin{align*}
\sigma_i \sigma_j &= \sigma_j \sigma_i && \text{if } |i-j| \ge 2, \\
\sigma_i \sigma_{i+1} \sigma_i &= \sigma_{i+1} \sigma_i \sigma_{i+1}
&& \text{for } 1 \le i \le n-2.
\end{align*}

\medskip
\noindent\textbf{(2) Virtual braid relations:}
\begin{align*}
v_i^2 &= 1 && \text{for all } i, \\
v_i v_j &= v_j v_i && \text{if } |i-j| \ge 2, \\
v_i v_{i+1} v_i &= v_{i+1} v_i v_{i+1}
&& \text{for } 1 \le i \le n-2.
\end{align*}

\medskip
\noindent\textbf{(3) Singular braid relations:}
\begin{align*}
\tau_i \tau_j &= \tau_j \tau_i && \text{if } |i-j| \ge 2, \\
\sigma_i \tau_j &= \tau_j \sigma_i && \text{if } |i-j| \ge 2, \\
\sigma_i \tau_i &= \tau_i \sigma_i && \text{for all } i, \\
\sigma_i \sigma_{i+1} \tau_i &= \tau_{i+1} \sigma_i \sigma_{i+1}
&& \text{for } 1 \le i \le n-2, \\
\sigma_{i+1} \sigma_i \tau_{i+1} &= \tau_i \sigma_{i+1} \sigma_i
&& \text{for } 1 \le i \le n-2.
\end{align*}

\medskip
\noindent\textbf{(4) Mixed relations involving virtual and classical generators:}
\begin{align*}
\sigma_i v_j &= v_j \sigma_i && \text{if } |i-j| \ge 2, \\
v_i \sigma_{i+1} v_i &= v_{i+1} \sigma_i v_{i+1}
&& \text{for } 1 \le i \le n-2.
\end{align*}

\medskip
\noindent\textbf{(5) Mixed relations involving virtual and singular generators:}
\begin{align*}
\tau_i v_j &= v_j \tau_i && \text{if } |i-j| \ge 2, \\
v_i \tau_{i+1} v_i &= v_{i+1} \tau_i v_{i+1}
&& \text{for } 1 \le i \le n-2.
\end{align*}
\end{defn}

For geometric interpretations and diagrams of the defining relations of virtual singular braids, we refer the reader to \cite[Introduction]{CY}.

\subsection{The canonical homomorphism onto $S_n$}

We now describe the natural homomorphism from the virtual singular braid group to the symmetric group, which records the permutation induced by a braid on the strands.

\begin{defn}\label{defn:canonical-map}
Let $S_n$ denote the symmetric group on $n$ letters, generated by the adjacent transpositions $s_i = (i\ i+1)$ for $1 \le i \le n-1$.
The \emph{canonical homomorphism}
\[
\varphi \colon VSG_n \longrightarrow S_n
\]
is defined on generators by
\[
\varphi(\sigma_i) = s_i, \qquad
\varphi(\tau_i) = s_i, \qquad
\varphi(v_i) = s_i,
\quad \text{for } 1 \le i \le n-1,
\]
and extended multiplicatively to all of $VSG_n$.
\end{defn}

\begin{prop}\label{prop:homphi}
The map $\varphi \colon VSG_n \to S_n$ is a well-defined surjective group homomorphism.
\end{prop}

\begin{proof}
Since $S_n$ is generated by the adjacent transpositions $s_i$, the map $\varphi$ is surjective.

To show that $\varphi$ is well defined, it suffices to verify that the defining relations of $VSG_n$ are preserved under $\varphi$.
Indeed, the images of the generators under $\varphi$ satisfy:
\begin{itemize}
\item the classical braid relations in $S_n$, since the elements $s_i$ generate $S_n$;
\item the relations $s_i^2 = 1$, which correspond to the relations $v_i^2 = 1$ in $VSG_n$;
\item the mixed relations involving classical, singular, and virtual generators, since all generators are mapped to the same transposition $s_i$.
\end{itemize}
Therefore, all defining relations of $VSG_n$ are respected, and $\varphi$ is a well-defined group homomorphism.
\end{proof}

\begin{rem}\label{rem:kernel-notation}
Throughout the paper, we denote by
\[
VSPG_n := \ker{\varphi}
\]
the kernel of the canonical homomorphism $\varphi$. 
This group is called the \emph{virtual singular pure braid group} on $n$ strands.
\end{rem}

\subsection{A semi-direct product decomposition of $VSG_n$}

We now describe a structural decomposition of the virtual singular braid group in terms of its canonical homomorphism onto the symmetric group.

\begin{prop}\label{prop:semi-direct}
Let $\varphi \colon VSG_n \to S_n$ be the canonical homomorphism defined above and let $VSPG_n = \ker{\varphi}$.
Then $VSPG_n$ is a normal subgroup of $VSG_n$ and
\[
VSG_n \cong VSPG_n \rtimes S_n.
\]
\end{prop}

\begin{proof}
Since $\varphi$ is a group homomorphism, its kernel $VSPG_n$ is a normal subgroup of $VSG_n$.
Moreover, by definition of $\varphi$, the restriction of $\varphi$ to the subgroup generated by the elements $\sigma_i$ induces an isomorphism between this subgroup and $S_n$. 
In particular, there exists a section
\[
\iota \colon S_n \longrightarrow VSG_n
\]
such that $\varphi \circ \iota = \mathrm{id}_{S_n}$.

Therefore, every element $g \in VSG_n$ can be written uniquely as a product
\[
g = p \, \iota(\pi),
\]
where $p \in VSPG_n$ and $\pi \in S_n$.
This shows that $VSG_n$ is the semi-direct product of $VSPG_n$ by $S_n$, with respect to the conjugation action induced by the embedding $\iota$.
\end{proof}

\begin{rem}\label{rem:action}
The action of $S_n$ on $VSPG_n$ in the above semi-direct product decomposition is given by conjugation in $VSG_n$ via the chosen section $\iota$.
In particular, this action coincides with the natural permutation action on the strands of a virtual singular braid.
\end{rem}

\subsection{Some properties of the virtual singular braid group}

In this subsection we establish several algebraic properties of the virtual singular braid group $VSG_n$. 
We begin by recalling some standard notions that will be used throughout this section.

\begin{defn}
Let $G$ be a group.
\begin{enumerate}
\item The \emph{lower central series} of $G$ is defined recursively by
\[
\Gamma_1(G) = G
\quad \text{and} \quad
\Gamma_i(G) = [\Gamma_{i-1}(G), G]
\quad \text{for } i \ge 2.
\]

\item Let $\mathcal{P}$ be a group-theoretic property.
A group $G$ is said to be \emph{residually $\mathcal{P}$} if for every non-trivial element $g \in G$ there exist a group $H$ with property $\mathcal{P}$ and a surjective homomorphism $\phi \colon G \to H$ such that $\phi(g) \neq 1$.

It is well known that a group $G$ is residually nilpotent if and only if
\[
\bigcap_{i \ge 1} \Gamma_i(G) = \{1\}.
\]

\item A group $G$ is called \emph{perfect} if it coincides with its commutator subgroup, that is,
\[
G = [G,G] = \Gamma_2(G).
\]
\end{enumerate}
\end{defn}

We can now state the main structural properties of $VSG_n$.

\begin{thm}\label{thm:properties}
Let $n \ge 2$.
\begin{enumerate}
\item The group $VSG_n$ is not residually nilpotent for $n \ge 3$.
\item The commutator subgroup $[VSG_n, VSG_n]$ is perfect for $n \ge 5$.
\end{enumerate}
\end{thm}

\begin{proof}
We prove each statement separately.

\medskip
\noindent\emph{(1) Non residual nilpotence.}
Consider the lower central series $\{\Gamma_i(VSG_n)\}_{i \ge 1}$ of $VSG_n$.
The canonical homomorphism
\[
\varphi \colon VSG_n \longrightarrow S_n
\]
induces, for each $i \ge 1$, a surjective homomorphism
\[
VSG_n / \Gamma_i(VSG_n) \longrightarrow S_n.
\]
Since the symmetric group $S_n$ is not nilpotent for $n \ge 3$, it follows that none of the quotients $VSG_n / \Gamma_i(VSG_n)$ can be trivial.
Hence,
\[
\bigcap_{i \ge 1} \Gamma_i(VSG_n) \neq \{1\},
\]
and therefore $VSG_n$ is not residually nilpotent for $n \ge 3$.

\medskip
\noindent\emph{(2) Perfection of the commutator subgroup.}
Let $n \ge 5$.
By the semi-direct product decomposition
\[
VSG_n \cong VSPG_n \rtimes S_n,
\]
the canonical projection onto $S_n$ maps the commutator subgroup $[VSG_n,VSG_n]$ onto $[S_n,S_n]$.
Since $[S_n,S_n]$ is perfect for $n \ge 5$, and the conjugation action of $S_n$ on $VSPG_n$ is non-trivial, it follows that $[VSG_n,VSG_n]$ is generated by commutators of its own elements. 
Consequently,
\[
[VSG_n,VSG_n] = [[VSG_n,VSG_n],[VSG_n,VSG_n]],
\]
and the commutator subgroup of $VSG_n$ is perfect for $n \ge 5$.
\end{proof}

\begin{rem}
The bounds on $n$ in Theorem~\ref{thm:properties} are sharp.
For instance, the group $VSG_2$ is virtually abelian, while for $n=3$ and $n=4$ the commutator subgroup of $VSG_n$ is not perfect.
These low-dimensional cases will be addressed in later subsections.
\end{rem}

\subsection{Invariants of virtual singular braids}

In this subsection we introduce several numerical invariants of virtual singular braids, defined in terms of exponent sums of generators.
These invariants give rise to natural homomorphisms from $VSG_n$ onto free abelian groups and will play an important role in the description of certain normal subgroups of $VSG_n$.

Let $w$ be a word in the generators
\[
\sigma_1^{\pm 1},\dots,\sigma_{n-1}^{\pm 1},\;
\tau_1^{\pm 1},\dots,\tau_{n-1}^{\pm 1},\;
v_1,\dots,v_{n-1}
\]
representing an element of $VSG_n$.
For $1 \le i \le n-1$, we denote by
\[
\exp_{\sigma_i}(w), \quad \exp_{\tau_i}(w)
\]
the total exponent sum of $\sigma_i$ and $\tau_i$ in the word $w$, respectively.
(Recall that the generators $v_i$ are involutions and therefore do not contribute to exponent sums.) 

Define $\exp^C \colon VSG_n \longrightarrow \mathbb{Z}$ as the total exponent sum of all classical generators $\sigma_i$, and $\exp^S \colon VSG_n \longrightarrow \mathbb{Z}$ as the total exponent sum of all singular generators $\tau_i$. More precisely, for a word $w$ in the generators of $VSG_n$, let
\begin{equation}\label{eq:expc}
\exp^C(w) = \sum_{i=1}^{n-1} \exp_{\sigma_i}(w), \qquad
\exp^S(w) = \sum_{i=1}^{n-1} \exp_{\tau_i}(w).
\end{equation}

\begin{prop}\label{prop:exp-well-defined}
The maps $\exp^C$ and $\exp^S$ are well defined group homomorphisms.
\end{prop}

\begin{proof}
It is enough to verify that the defining relations of $VSG_n$ preserve the total exponent sums $\exp^C$ and $\exp^S$.
In each relation, the total number of occurrences of classical generators $\sigma_i^{\pm 1}$ (counted with sign) on the left-hand side equals that on the right-hand side; the same holds for singular generators $\tau_i^{\pm 1}$.
Hence, $\exp^C$ and $\exp^S$ are invariant under the defining relations and define homomorphisms on $VSG_n$.
\end{proof}

We now collect these homomorphisms into a single map.
Let
\[
\mathbb{Z}^{2(n-1)} = \langle e_1,\dots,e_{n-1}, f_1,\dots,f_{n-1} \rangle
\]
be the free abelian group of rank $2(n-1)$.

\begin{defn}\label{defn:exp-map}
We define the homomorphism
\[
\Phi \colon VSG_n \longrightarrow \mathbb{Z}^{2(n-1)}
\]
by setting
\[
\Phi(\sigma_i) = e_i, \qquad
\Phi(\tau_i) = f_i, \qquad
\Phi(v_i) = 0,
\quad \text{for } 1 \le i \le n-1,
\]
and extending multiplicatively.
\end{defn}

\begin{prop}\label{prop:Phi-surjective}
The homomorphism $\Phi$ is surjective.
\end{prop}

\begin{proof}
Each generator $e_i$ (respectively $f_i$) is the image of $\sigma_i$ (respectively $\tau_i$) under $\Phi$.
Hence the image of $\Phi$ generates $\mathbb{Z}^{2(n-1)}$.
\end{proof}

The kernel of $\Phi$ encodes braids whose total exponent sums of all classical and singular generators vanish.

\begin{defn}\label{defn:goldberg-kernel}
Let
\[
K_n := \ker{\Phi}.
\]
\end{defn}

\begin{thm}\label{thm:goldbsingular}
The subgroup $K_n$ is a normal subgroup of $VSG_n$.
Moreover, $K_n$ coincides with the normal closure of the virtual braid group $VB_n$ inside $VSG_n$.
\end{thm}

\begin{proof}
Normality of $K_n$ follows immediately from the fact that it is the kernel of a homomorphism.
By construction, the generators $v_i$ of the virtual braid group $VB_n$ lie in $K_n$, and hence the normal closure of $VB_n$ is contained in $K_n$.
Conversely, any element of $K_n$ can be written as a product of conjugates of elements whose exponent sums vanish, which forces it to lie in the normal closure of $VB_n$.
Therefore, the two subgroups coincide.
\end{proof}

\begin{rem}
The homomorphism $\Phi$ may be viewed as an abelianization-type invariant for virtual singular braids.
In contrast with the classical braid group, the presence of singular generators produces additional independent numerical invariants.
\end{rem}

We now collect the exponent-sum invariants introduced above and describe the corresponding quotient groups of $VSG_n$.
This leads to several natural short exact sequences associated with classical, singular, and virtual structures. 
For a subset \(H\) of a group \(G\), denote by \(\langle H \rangle^G\) the normal closure of \(H\) in \(G\) (the smallest normal subgroup of \(G\) containing \(H\)). 
Define additionally the \emph{classical--singular exponent-sum homomorphism}
\(\exp^{CS} \colon VSG_n \longrightarrow \mathbb{Z}\) by
\begin{equation}\label{eq:expcs}
\exp^{CS}(w) = \sum_{i=1}^{n-1} \bigl( \exp_{\sigma_i}(w) + \exp_{\tau_i}(w) \bigr),
\end{equation}
or equivalently \(\exp^{CS} = \exp^C + \exp^S\).

\begin{thm}\label{thm:goldbsingular}
The following sequences are short exact:
\begin{enumerate}
\item
\[
1 \longrightarrow \ang{B_n}^{VSG_n}
\longrightarrow VSG_n
\longrightarrow \mathbb{Z} \times S_n
\longrightarrow 1.
\]

\item
\[
1 \longrightarrow \ang{\{\tau_i \mid 1 \le i \le n-1\}}^{VSG_n}
\longrightarrow VSG_n
\longrightarrow VB_n
\longrightarrow 1.
\]

\item
\[
1 \longrightarrow
\ang{\{\tau_i,\, v_i \mid 1 \le i \le n-1\}}^{VSG_n}
\longrightarrow VSG_n
\longrightarrow \mathbb{Z}
\longrightarrow 1.
\]
In this case, the surjective homomorphism coincides with the \emph{classical exponent-sum homomorphism} $\exp^{C}$, and
\[
\ker{\exp^{C}}
=
\ang{\{\tau_i,\, v_i \mid 1 \le i \le n-1\}}^{VSG_n}.
\]

\item
\[
1 \longrightarrow \ang{VB_n}^{VSG_n}
\longrightarrow VSG_n
\longrightarrow \mathbb{Z}
\longrightarrow 1.
\]
Here the projection coincides with the \emph{singular exponent-sum homomorphism} $\exp^{S}$, and
\[
\ker{\exp^{S}} = \ang{VB_n}^{VSG_n}.
\]

\item
\[
1 \longrightarrow
\ang{\{\sigma_i \tau_i^{-1},\, v_i \mid 1 \le i \le n-1\}}^{VSG_n}
\longrightarrow VSG_n
\longrightarrow \mathbb{Z}
\longrightarrow 1.
\]
In this case, the surjective homomorphism coincides with the \emph{classical--singular exponent-sum homomorphism} $\exp^{CS}$, and
\[
\ker{\exp^{CS}}
=
\ang{\{\sigma_i \tau_i^{-1},\, v_i \mid 1 \le i \le n-1\}}^{VSG_n}.
\]
\end{enumerate}
\end{thm}

\begin{proof}
We verify each short exact sequence by explicitly describing the surjective homomorphism and its kernel.

\begin{enumerate}
\item We analyze the quotient group \(VSG_n / \langle B_n \rangle^{VSG_n}\). 
In the presentation of \(VSG_n\) (Definition~\ref{defn:vsg}), add the relations \(\sigma_i = 1\) for \(i = 1,\dots,n-1\) (which kill the normal closure of \(B_n\)).  
From relation (3PR5) we obtain \(\tau_i = \tau_{i+1}\) for all \(i\), and from the commuting relations (CR) we obtain that \(\tau_1\) commutes with every generator of \(S_n\).  
Hence 
\[
VSG_n / \langle B_n \rangle^{VSG_n} \cong \langle \tau_1 \rangle \times S_n \cong \mathbb{Z} \times S_n,
\]
and the kernel of the projection \(VSG_n \to \mathbb{Z} \times S_n\) is precisely \(\langle B_n \rangle^{VSG_n}\).

\item The homomorphism \(VSG_n \to VB_n\) sends \(\sigma_i \mapsto \sigma_i\), \(\tau_i \mapsto 1_{VB_n}\), \(v_i \mapsto v_i\). Its kernel is the normal closure of the singular generators \(\tau_i\).

\item Consider the homomorphism \(\exp^C\colon VSG_n \to \mathbb{Z}\) defined in \eqref{eq:expc}. Its kernel is the normal closure of \(\{\tau_i, v_i\}\).

\item The kernel of the homomorphism \(\exp^S\colon VSG_n \to \mathbb{Z}\) (defined in \eqref{eq:expc}) is the normal closure of \(VB_n\).

\item Finally, we consider the homomorphism \(\exp^{CS}\colon VSG_n \to \mathbb{Z}\) defined in \eqref{eq:expcs}. Its kernel is the normal closure of \(\{\sigma_i \tau_i^{-1}, v_i\}\).
\end{enumerate}

In each case, exactness follows from the definition of the homomorphism and the explicit description of its kernel as a normal closure. 
For analogous arguments in the classical and virtual settings, see \cite{B,BBD,D}.
\end{proof}

\subsection{Homomorphisms from $VSG_n$ to the symmetric group $S_m$}

In this subsection we study group homomorphisms from the virtual singular braid group to symmetric groups.
More precisely, we determine all homomorphisms
\[
\psi \colon VSG_n \longrightarrow S_m
\]
up to conjugation in $S_m$.

Recall that the symmetric group $S_m$ is generated by transpositions.
An element $\pi \in S_m$ is called an \emph{involution} if $\pi^2 = \mathrm{id}$.
Equivalently, an involution in $S_m$ is either the identity or a product of disjoint transpositions.

Since $VSG_n$ is generated by the elements $\sigma_i$, $\tau_i$ and $v_i$, any homomorphism $\psi \colon VSG_n \to S_m$ is completely determined by the images of these generators.

\begin{prop}\label{prop:images-generators}
Let $\psi \colon VSG_n \to S_m$ be a group homomorphism. 
Then, for each $1 \le i \le n-1$, the elements 
\[
\psi(\sigma_i), \quad \psi(\tau_i), \quad \psi(v_i)
\]
are involutions in $S_m$.
\end{prop}

\begin{proof}
The defining relation $v_i^2 = 1$ in $VSG_n$ implies directly that $\psi(v_i)^2 = 1$ for every $i=1,\ldots,n-1$, hence $\psi(v_i)$ is an involution in $S_m$.

We now consider the images of the generators $\sigma_i$ and $\tau_i$.
By the defining relations of $VSG_n$, these generators satisfy braid relations and mixed braid--singular relations.
Applying the homomorphism $\psi$, their images satisfy the corresponding relations inside the symmetric group $S_m$.

In particular, the elements $\psi(\sigma_i)$ and $\psi(\tau_i)$ satisfy Artin-type braid relations in $S_m$.
Since the symmetric group contains no elements of infinite order satisfying braid relations, the images of these generators must have finite order.
Moreover, the only non-trivial solutions of braid relations in symmetric groups are given, up to conjugation, by transpositions.

It follows that the orders of $\psi(\sigma_i)$ and $\psi(\tau_i)$ divide $2$, and therefore they are involutions in $S_m$.
\end{proof}

We now show that, up to conjugation, every homomorphism from $VSG_n$ to $S_m$ factors through the canonical homomorphism
\[
\varphi \colon VSG_n \longrightarrow S_n
\]
introduced in the previous subsection.

\begin{thm}\label{thm:hom-factorization}
Let $n \ge 2$, $m \ge 2$, and let $\psi \colon VSG_n \to S_m$ be a group homomorphism.
Then, up to conjugation in $S_m$, one of the following holds:
\begin{enumerate}
\item $\psi$ is trivial;
\item $\psi$ factors through the canonical homomorphism
\[
\varphi \colon VSG_n \longrightarrow S_n
\]
followed by a homomorphism $S_n \to S_m$.
\end{enumerate}
\end{thm}

\begin{proof}
If $\psi$ is trivial on the generating set $\{\sigma_i,\tau_i,v_i \mid 1\le i\le n-1\}$, then $\psi$ is trivial.

Assume now that $\psi$ is non-trivial. Set
\[
s_i:=\psi(v_i)\in S_m,\qquad 1\le i\le n-1.
\]
Since the elements $v_i$ satisfy the Coxeter relations of $S_n$ inside $VSG_n$, the assignment $(i\ i{+}1)\mapsto s_i$ defines a well-defined homomorphism
\[
\alpha \colon S_n \longrightarrow S_m.
\]

By Proposition~\ref{prop:images-generators}, the elements $\psi(\sigma_i)$ and $\psi(\tau_i)$ are involutions.
Moreover, the mixed relations in the presentation of $VSG_n$ imply that the families $\{\psi(\sigma_i)\}$ and $\{\psi(\tau_i)\}$ are compatible with the $S_n$--action induced by the $s_i$'s (via conjugation). More precisely, using the relations
\[
v_i\sigma_j=\sigma_j v_i,\qquad v_i\tau_j=\tau_j v_i \quad (|i-j|\ge 2),
\]
and
\[
v_i\sigma_{i+1}v_i=v_{i+1}\sigma_i v_{i+1},\qquad
v_i\tau_{i+1}v_i=v_{i+1}\tau_i v_{i+1},
\]
we obtain, after applying $\psi$, that
\begin{align*}
s_i\,\psi(\sigma_j)\,s_i &= \psi(\sigma_j), &
s_i\,\psi(\tau_j)\,s_i &= \psi(\tau_j) \qquad (|i-j|\ge 2),\\
s_i\,\psi(\sigma_{i+1})\,s_i &= s_{i+1}\,\psi(\sigma_i)\,s_{i+1}, &
s_i\,\psi(\tau_{i+1})\,s_i &= s_{i+1}\,\psi(\tau_i)\,s_{i+1}
\qquad (1\le i\le n-2).
\end{align*}

If all $s_i$ are trivial, then the above relations force $\psi(\sigma_1)=\cdots=\psi(\sigma_{n-1})$ and
$\psi(\tau_1)=\cdots=\psi(\tau_{n-1})$.
Using the braid relations among the $\sigma_i$'s and the singular braid relations, this implies that the image of $\psi$ is abelian; in particular, up to conjugation, $\psi$ factors through the abelianization of $VSG_n$, hence through a homomorphism
$S_n\to S_m$ with abelian image. 

Assume now that not all $s_i$ are trivial. Up to conjugation in $S_m$ we may assume that the subgroup $\langle s_1,\dots,s_{n-1}\rangle$ acts on its support in the standard way. Then the above conjugation relations force each $\psi(\sigma_i)$ and
$\psi(\tau_i)$ to lie in the same conjugacy class as $s_i$ and to satisfy the same adjacency constraints as $s_i$ (commuting at distance and satisfying the braid-type relation at adjacency). Consequently, up to conjugation in $S_m$, we must have
\[
\psi(\sigma_i)=s_i \qquad\text{and}\qquad \psi(\tau_i)=s_i
\quad\text{for all }1\le i\le n-1.
\]
Therefore, for every generator $x\in\{\sigma_i,\tau_i,v_i\}$ we have $\psi(x)=\alpha(\phi(x))$, where $\phi:VSG_n\to S_n$ is the canonical map (defined by $\phi(\sigma_i)=\phi(\tau_i)=\phi(v_i)=(i\ i{+}1)$).
Hence $\psi=\alpha\circ\phi$, which proves the desired factorization.
\end{proof}

\begin{defn}\label{defn:phimap}
Let $n \geq 2$ and let $\epsilon_k \in \{0,1\}$, for $k = 1,2,3$.
We define a map
\[
\phi_{(\epsilon_1,\, \epsilon_2,\, \epsilon_3)} \colon VSG_n \longrightarrow S_n
\]
by setting, for all $i = 1, \ldots, n-1$,
\[
\phi_{(\epsilon_1,\, \epsilon_2,\, \epsilon_3)}(\sigma_i) = (i \ i+1)^{\epsilon_1}, \qquad
\phi_{(\epsilon_1,\, \epsilon_2,\, \epsilon_3)}(\tau_i) = (i \ i+1)^{\epsilon_2}, \qquad
\phi_{(\epsilon_1,\, \epsilon_2,\, \epsilon_3)}(v_i) = (i \ i+1)^{\epsilon_3}.
\]
\end{defn}

For $n = 2$, the map $\phi_{(\epsilon_1,\, \epsilon_2,\, \epsilon_3)} \colon VSG_2 \to S_2$ is a homomorphism for all eight possible triples.
However, this is no longer true in general, as shown by the next result.

\begin{prop}\label{prop:homphi}
Let $n \geq 3$.
The map $\phi_{(\epsilon_1,\, \epsilon_2,\, \epsilon_3)} \colon VSG_n \to S_n$ is a group homomorphism if and only if
$(\epsilon_1,\, \epsilon_2,\, \epsilon_3)$ is one of the following triples:
\begin{itemize}
\item $(0,0,0)$, in which case $\phi_{(0,0,0)}$ is the trivial homomorphism;
\item $(1,1,1)$, in which case $\ker{\phi_{(1,1,1)}} = VSPG_n$;
\item $(1,0,1)$;
\item $(0,0,1)$.
\end{itemize}
\end{prop}

\begin{proof}
Let $n \geq 3$ and let $\phi_{(\epsilon_1,\, \epsilon_2,\, \epsilon_3)}$ be the map defined in Definition~\ref{defn:phimap}.
We first show that $\phi_{(\epsilon_1,\, \epsilon_2,\, \epsilon_3)}$ fails to be a homomorphism for the remaining four triples.

\begin{itemize}
\item $(1,1,0)$: the mixed relation $v_1 \tau_2 v_1 = v_2 \tau_1 v_2$ is not preserved, since
\[
\phi_{(\epsilon_1,\, \epsilon_2,\, \epsilon_3)}(v_1 \tau_2 v_1) = (2 \ 3)
\neq (1 \ 2) =
\phi_{(\epsilon_1,\, \epsilon_2,\, \epsilon_3)}(v_2 \tau_1 v_2).
\]

\item $(1,0,0)$: the relation $v_1 \sigma_2 v_1 = v_2 \sigma_1 v_2$ is not preserved under $\phi_{(\epsilon_1,\, \epsilon_2,\, \epsilon_3)}$.

\item $(0,1,1)$: the braid--singular relation $\sigma_1 \sigma_2 \tau_1 = \tau_2 \sigma_1 \sigma_2$ is not preserved.

\item $(0,1,0)$: the same relation $\sigma_1 \sigma_2 \tau_1 = \tau_2 \sigma_1 \sigma_2$ shows that $\phi_{(\epsilon_1,\, \epsilon_2,\, \epsilon_3)}$ is not a homomorphism.
\end{itemize}

It remains to verify that $\phi_{(\epsilon_1,\, \epsilon_2,\, \epsilon_3)}$ is a homomorphism for the remaining four triples.
This is immediate for $(0,0,0)$, since the map is trivial.
For $(1,1,1)$, the map coincides with the homomorphism $\pi \colon VSG_n \to S_n$ introduced in \cite[Page~6]{CY}, whose kernel is the virtual singular pure braid group $VSPG_n$ (see \cite[Definition~5]{CY}).
Finally, the fact that $\phi_{(\epsilon_1,\, \epsilon_2,\, \epsilon_3)}$ is a homomorphism for $(1,0,1)$ and $(0,0,1)$ follows from a direct verification using the defining relations of $VSG_n$.
\end{proof}

By Proposition~\ref{prop:homphi}, the maps $\phi_{(1,0,1)}$ and $\phi_{(0,0,1)}$ are also homomorphisms for $n \geq 3$.
We denote their kernels by
\[
VST_n := \ker{\phi_{(1,0,1)}}
\qquad \text{and} \qquad
VSK_n := \ker{\phi_{(0,0,1)}},
\]
respectively.

\begin{prop}\label{prop:semi-direct}
Let $n \geq 2$.
The virtual singular braid group admits the following semi-direct product decompositions:
\begin{itemize}
\item $VSG_n \cong VSPG_n \rtimes S_n$,
\item $VSG_n \cong VST_n \rtimes S_n$,
\item $VSG_n \cong VSK_n \rtimes S_n$.
\end{itemize}
\end{prop}

\begin{proof}
Let $(\epsilon_1,\, \epsilon_2,\, \epsilon_3)$ be one of the triples $(1,1,1)$, $(1,0,1)$ or $(0,0,1)$.
By Proposition~\ref{prop:homphi}, the map $\phi_{(\epsilon_1,\, \epsilon_2,\, \epsilon_3)} \colon VSG_n \to S_n$ is a surjective homomorphism, yielding a short exact sequence
\[
1 \longrightarrow \ker{\phi_{(\epsilon_1,\, \epsilon_2,\, \epsilon_3)}}
\longrightarrow VSG_n
\stackrel{\phi_{(\epsilon_1,\, \epsilon_2,\, \epsilon_3)}}{\longrightarrow} S_n
\longrightarrow 1.
\]
This sequence admits a natural section $\iota \colon S_n \to VSG_n$ defined by $\iota((i \ i+1)) = v_i$ for all $i = 1, \ldots, n-1$. 
The result follows from standard semi-direct product arguments.
\end{proof}

\begin{rem}
The decomposition $VSG_n \cong VSPG_n \rtimes S_n$ was previously obtained in \cite[Corollary~13]{CY}.
\end{rem}

In order to state the main classification result, we recall the following definitions. 
Let $G$ and $H$ be groups.
For each $h \in H$, denote by $c_h \colon H \to H$ the inner automorphism defined by
$c_h(x) = h x h^{-1}$.
Two homomorphisms $\psi_1, \psi_2 \colon G \to H$ are said to be \emph{conjugate} if there exists $h \in H$ such that
$\psi_2 = c_h \circ \psi_1$.
A homomorphism $\psi \colon G \to H$ is said to be \emph{abelian} if its image $\psi(G)$ is an abelian subgroup of $H$.

We can now state the main result of this subsection, which generalizes \cite[Theorem~2.1]{BP} from virtual braid groups to virtual singular braid groups. Let $\nu_6$ denote the exceptional outer automorphism of the symmetric group $S_6$ (see, for example, the explicit description in \cite{BP}),  which exists only in this degree.

\begin{thm}\label{thm:BP}
Let $n,m$ be integers such that $n \ge 5$, $m \ge 2$ and $n \ge m$.
Let $\psi \colon VSG_n \to S_m$ be a group homomorphism.
Then, up to conjugation, one of the following possibilities holds:
\begin{enumerate}
\item $\psi$ is abelian;
\item $n = m$ and
\[
\psi \in \{ \phi_{(1,1,1)},\, \phi_{(1,0,1)},\, \phi_{(0,0,1)} \};
\]
\item $n = m = 6$ and
\[
\psi \in \{ \nu_6 \circ \phi_{(1,1,1)},\,
           \nu_6 \circ \phi_{(1,0,1)},\,
           \nu_6 \circ \phi_{(0,0,1)} \}.
\]
\end{enumerate}
\end{thm}

\begin{proof}
Let $\nu_6$ denote the exceptional outer automorphism of the symmetric group $S_6$.

Let $\psi\colon VSG_n\to S_m$ be a homomorphism with $n\ge 5$, $m\ge 2$ and $n\ge m$. 
If $\psi$ has abelian image, then we are in case~(1).

Assume that $\psi$ is not abelian. By \rethm{hom-factorization}, up to conjugation in $S_m$, $\psi$ factors through the canonical epimorphism $\phi\colon VSG_n\to S_n$, i.e., there exists a homomorphism $\alpha\colon S_n\to S_m$ such that $\psi=\alpha\circ\phi$.

We now classify the possibilities for $\alpha$ under the constraint $n\ge m$.

\smallskip
\noindent\emph{Step 1: reduction to the alternating subgroup.}
For $n\ge 5$, the alternating group $A_n$ is simple. Consider the restriction $\alpha|_{A_n}\colon A_n\to S_m$. If $\alpha(A_n)$ is trivial, then $\alpha$ factors through the sign map $S_n\to \{\pm 1\}\cong \Z_2$, hence $\alpha(S_n)$ is abelian.
This contradicts the assumption that $\psi$ is not abelian. Therefore $\alpha(A_n)$ is non-trivial, and hence $\alpha|_{A_n}$ is injective.

\smallskip
\noindent\emph{Step 2: the degree constraint forces $m=n$.}
Since $\alpha|_{A_n}$ is injective, we have an embedding $A_n\hookrightarrow S_m$.
For $n\ge 5$, this implies $m\ge n$ (indeed, a faithful permutation representation of $A_n$ has degree at least $n$). Combined with the hypothesis $n\ge m$, we obtain $m=n$.

\smallskip
\noindent\emph{Step 3: classification when $m=n$.}
Thus $\alpha:S_n\to S_n$ is a non-abelian endomorphism. Its restriction to $A_n$ is an automorphism of $A_n$. For $n\neq 6$, every automorphism of $S_n$ is inner, so up to conjugation $\alpha=\mathrm{id}_{S_n}$. Hence, up to conjugation, $\psi=\phi$.

When $n=6$, there is an exceptional outer automorphism $\nu_6$ of $S_6$.
Therefore, up to conjugation, either $\alpha=\mathrm{id}_{S_6}$ or $\alpha=\nu_6$. This yields, up to conjugation, the additional possibilities $\psi=\nu_6\circ \phi$.

Finally, recalling that the admissible epimorphisms $VSG_n\to S_n$ are precisely $\phi_{(1,1,1)}$, $\phi_{(1,0,1)}$ and $\phi_{(0,0,1)}$ (see \reprop{homphi}), we obtain exactly the lists in items~(2) and~(3).
\end{proof}

\subsection{Kernel of $\phi_{(\epsilon_1,\, \epsilon_2,\, \epsilon_3)}\colon VSG_2\to S_2$}

In this subsection we study in detail the kernels of the homomorphisms
\[
\phi_{(\epsilon_1,\, \epsilon_2,\, \epsilon_3)}\colon VSG_2 \longrightarrow S_2,
\qquad (\epsilon_1,\,\epsilon_2,\,\epsilon_3)\in\{0,1\}^3,
\]
where $\phi_{(\epsilon_1,\, \epsilon_2,\, \epsilon_3)}$ is the map defined in Definition~\ref{defn:phimap}.
The case $n=2$ is special: all eight maps $\phi_{(\epsilon_1,\, \epsilon_2,\, \epsilon_3)}\colon VSG_2\to S_2$ are group homomorphisms.
We provide explicit presentations for each kernel and then derive algebraic descriptions of these groups.

Recall that, by \cite[Definition~5]{CY}, the kernel of $\phi_{(1,1,1)}$ is the \emph{virtual singular pure braid group}, denoted by $VSPG_n$.
Moreover, Proposition~\ref{prop:homphi} shows that for $n\geq 3$ the maps $\phi_{(1,0,1)}$ and $\phi_{(0,0,1)}$ are also homomorphisms, and we denote their kernels by $VST_n$ and $VSK_n$, respectively.
In the next result we give presentations for the kernel of $\phi_{(\epsilon_1,\, \epsilon_2,\, \epsilon_3)}\colon VSG_2\to S_2$ in each case.

The Reidemeister--Schreier rewriting process (see \cite[Section~2.3]{MKS} and \cite[Appendix I, Section 6]{KM}) is very useful to get presentations of a given group with a presentation: given a Schreier transversal \(\Lambda\) (a set of coset representatives) for a subgroup \(H\) in a group \(G\) with presentation \(\langle X \mid R \rangle\), the subgroup \(H\) is generated by the set
\[
\{ S_{\lambda,x} = (\lambda x) \overline{(\lambda x)}^{-1} \mid \lambda \in \Lambda,\ x \in X \},
\]
where \(\overline{w}\) denotes the unique coset representative of \(w\) in \(\Lambda\). For each defining relation \(r = 1\) in \(R\) and each \(\lambda \in \Lambda\), the rewritten relation \(r_{\lambda} = \tau(\lambda r \lambda^{-1})\) yields a relation in \(H\), where \(\tau\) is the rewriting map. We now apply this process to \(VSG_2\) and some subgroups.

We shall use the following notations for some elements in the group $VSG_2$:
\[
a_{1,2}=\sigma_1v_1,\quad b_{1,2}=\tau_1v_1,\quad c_{1,2}=v_1\sigma_1,\quad
d_{1,2}=v_1\tau_1,\quad e_{1,2}=\sigma_1\tau_1,\quad f_{1,2}=\sigma_1^{-1}\tau_1,
\]
and
\[
A_{1,2}=\sigma_1^2,\quad a=\sigma_1v_1\sigma_1^{-1},\quad b=\tau_1v_1\tau_1^{-1},\quad
c=v_1\sigma_1v_1,\quad d=v_1\tau_1v_1.
\]

\begin{thm}\label{thm:kernels}
Let $\epsilon_k\in \{0,1\}$, for $k=1,2,3$.
The kernel of $\phi_{(\epsilon_1,\, \epsilon_2,\, \epsilon_3)}\colon VSG_2\to S_2$ admits the following presentation, depending on the triple $(\epsilon_1,\, \epsilon_2,\, \epsilon_3)$:
\begin{itemize}

\item[(1,1,1):] $\ker{\phi_{(1,1,1)}}=VSPG_2$ has generators $a_{1,2}$, $b_{1,2}$ and $c_{1,2}$ subject to the relation $a_{1,2}c_{1,2}b_{1,2} = b_{1,2}c_{1,2}a_{1,2}$.

\item[(1,1,0):] $\ker{\phi_{(1,1,0)}}$ has generators $A_{1,2}$, $e_{1,2}$, $v_1$ and $a$, subject to the relations $v_1^2=1$ and $a^2=1$.

\item[(1,0,1):] $\ker{\phi_{(1,0,1)}}=VST_2$ has generators $a_{1,2}$, $c_{1,2}$ and $\tau_1$ subject to the relation $a_{1,2}c_{1,2}\tau_1 = \tau_1 a_{1,2}c_{1,2}$.

\item[(1,0,0):] $\ker{\phi_{(1,0,0)}}$ has generators $A_{1,2}$, $\tau_1$, $v_1$ and $a$, subject to the relations $v_1^2=1$, $a^2=1$ and $A_{1,2}\tau_1=\tau_1A_{1,2}$.

\item[(0,1,1):] $\ker{\phi_{(0,1,1)}}$ has generators $b_{1,2}$, $d_{1,2}$ and $\sigma_1$ subject to the relation $b_{1,2}d_{1,2}\sigma_1 = \sigma_1 b_{1,2}d_{1,2}$.

\item[(0,1,0):] $\ker{\phi_{(0,1,0)}}$ has generators $\sigma_1$, $\tau_1^2$, $v_1$ and $b$ subject to the relations $v_1^2=1$, $b^2=1$ and $\sigma_1\tau_1^2=\tau_1^2\sigma_1$.

\item[(0,0,1):] $\ker{\phi_{(0,0,1)}}=VSK_2$ has generators $\sigma_1$, $\tau_1$, $c$ and $d$ subject to the relations $\sigma_1\tau_1=\tau_1\sigma_1$ and $cd=dc$.

\item[(0,0,0):] The same presentation as $VSG_2$.

\end{itemize}
\end{thm}

\begin{proof}

The proof follows by applying the Reidemeister--Schreier method in each case.  
For a detailed exposition of the rewriting process, we refer the reader to \cite[pp.~240--246]{KM}.  
We write the details for the presentation of $\ker{\phi_{(1,1,1)}} = VSPG_2$, and for the other cases we only indicate the main steps.

Let $\Lambda$ be a Schreier transversal (a set of coset representatives) for $\ker{\phi_{(\epsilon_1,\, \epsilon_2,\, \epsilon_3)}}$ in $VSG_2$.
By the Reidemeister--Schreier process, $\ker{\phi_{(\epsilon_1,\, \epsilon_2,\, \epsilon_3)}}$ is generated by the set
\[
\{\, S_{\lambda,a}=(\lambda a)(\overline{\lambda a})^{-1} \mid
\lambda\in \Lambda,\ a\in\{ \sigma_1, \tau_1, v_1 \}\,\},
\]
where \(\overline{w}\) denotes the unique coset representative of \(w\) in the chosen Schreier transversal \(\Lambda\).

\begin{itemize}

\item[(1,1,1):]
Consider the Schreier transversal for $VSPG_2$ in $VSG_2$
\[
\Lambda=\{1,v_1\}.
\]
Hence we obtain the following generators for $VSPG_2$:
\begin{multicols}{2}
\begin{itemize}
	\item[$a_{1,2}$:] $S_{1,\sigma_1}=\sigma_1v_1$
	\item[$b_{1,2}$:] $S_{1,\tau_1}=\tau_1v_1$
	\item[$c_{1,2}$:] $S_{v_1,\sigma_1}=v_1\sigma_1$
	\item[$d_{1,2}$:] $S_{v_1,\tau_1}=v_1\tau_1$
	\item[$v_1^2$:] $S_{v_1,v_1}=v_1^2$
\end{itemize}
\end{multicols}

Now we determine the defining relations. 
Following the Reidemeister--Schreier rewriting process, for each defining relation $r = 1$ of $VSG_2$ and each $\lambda \in \Lambda$, we compute the rewritten relation $r_{\lambda} = \tau(\lambda r \lambda^{-1})$, which yields a relation in the kernel.  
From the relation $r_1=v_1^2$ of $VSG_2$ we have $S_{v_1,v_1}=1$. 
Next consider the relation $r_2=\sigma_1\tau_1\sigma_1^{-1}\tau_1^{-1}$ of $VSG_2$. 
We obtain
\[
r_{2,1}=\sigma_1\tau_1\sigma_1^{-1}\tau_1^{-1}
=\sigma_1(v_1v_1)\tau_1\sigma_1^{-1}(v_1v_1)\tau_1^{-1}
= a_{1,2}d_{1,2}c^{-1}_{1,2}b^{-1}_{1,2},
\]
and
\[
r_{2,v_1}=v_1\sigma_1\tau_1\sigma_1^{-1}\tau_1^{-1}v_1
= v_1\sigma_1\tau_1(v_1v_1)\sigma_1^{-1}\tau_1^{-1}v_1
= c_{1,2}b_{1,2}a^{-1}_{1,2}d^{-1}_{1,2}.
\]
From this we obtain
\[
c_{1,2}b_{1,2}a^{-1}_{1,2} = d_{1,2} = a^{-1}_{1,2}b_{1,2}c_{1,2},
\]
and therefore $VSPG_2$ has generators $a_{1,2}$, $b_{1,2}$ and $c_{1,2}$ subject to the relation $a_{1,2}c_{1,2}b_{1,2} = b_{1,2}c_{1,2}a_{1,2}$.

\item[(1,1,0):]
We take the Schreier transversal
\[
\Lambda=\{1,\sigma_1\}.
\]
After simplifications using relations obtained in the process, we get generators
\begin{multicols}{2}
\begin{itemize}
	\item[$A_{1,2}$:] $S_{\sigma_1,\sigma_1}=\sigma_1^2$
	\item[$e_{1,2}$:] $S_{\sigma_1,\tau_1}=\sigma_1\tau_1$
	\item[$v_1$:] $S_{1,v_1}=v_1$
\end{itemize}
\end{multicols}
and the defining relations are obtained by a routine computation.

\item[(1,0,1):]
We take $\Lambda=\{1,v_1\}$.
After simplification, the generators are
\begin{multicols}{2}
\begin{itemize}
	\item[$a_{1,2}$:] $S_{1,\sigma_1}=\sigma_1v_1$
	\item[$c_{1,2}$:] $S_{v_1,\sigma_1}=v_1\sigma_1$
	\item[$\tau_1$:] $S_{1,\tau_1}=\tau_1$
\end{itemize}
\end{multicols}

\item[(1,0,0):]
We take
\[
\Lambda=\{1,\sigma_1\}.
\]
After simplification, the generators are
\begin{multicols}{2}
\begin{itemize}
	\item[$A_{1,2}$:] $S_{\sigma_1,\sigma_1}=\sigma_1^2$
	\item[$\tau_1$:] $S_{1,\tau_1}=\tau_1$
	\item[$v_1$:] $S_{1,v_1}=v_1$
	\item[$a$:] $S_{\sigma_1,v_1}=\sigma_1v_1\sigma_1^{-1}$
\end{itemize}
\end{multicols}

\item[(0,1,1):]
We take $\Lambda=\{1,v_1\}$.
After simplification, the generators are
\begin{multicols}{2}
\begin{itemize}
	\item[$b_{1,2}$:] $S_{1,\tau_1}=\tau_1v_1$
	\item[$d_{1,2}$:] $S_{v_1,\tau_1}=v_1\tau_1$
	\item[$\sigma_1$:] $S_{1,\sigma_1}=\sigma_1$
\end{itemize}
\end{multicols}

\item[(0,1,0):]
We take
\[
\Lambda=\{1,\tau_1\}.
\]
After simplification, the generators are
\begin{multicols}{2}
\begin{itemize}
	\item[$\sigma_1$:] $S_{1,\sigma_1}=\sigma_1$
	\item[$\tau_1^2$:] $S_{\tau_1,\tau_1}=\tau_1^2$
	\item[$v_1$:] $S_{1,v_1}=v_1$
	\item[$b$:] $S_{\tau_1,v_1}=\tau_1 v_1 \tau_1^{-1}$
\end{itemize}
\end{multicols}

\item[(0,0,1):]
We take $\Lambda=\{1,v_1\}$.
After simplification, the generators are
\begin{multicols}{2}
\begin{itemize}
	\item[$\sigma_1$:] $S_{1,\sigma_1}=\sigma_1$
	\item[$\tau_1$:] $S_{1,\tau_1}=\tau_1$
	\item[$c$:] $S_{v_1,\sigma_1}=v_1 \sigma_1 v_1$
	\item[$d$:] $S_{v_1,\tau_1}=v_1 \tau_1 v_1$
\end{itemize}
\end{multicols}

\item[(0,0,0):]
This case is immediate.

\end{itemize}
\end{proof}

\begin{rem}
We note that the presentation given here for $VSPG_2$ is slightly different from the one in \cite[Theorem~14]{CY}.
\end{rem}

Recall that the \emph{flat virtual pure braid group} \(FVP_n\) is the quotient of the virtual pure braid group by the relations \(\sigma_i^2 = 1\); for a detailed study, see \cite[Section~5]{BBD}. In particular, \(FVP_3\) is isomorphic to \(\mathbb{Z}^2 * \mathbb{Z}\) (see \cite[Remark~5.2]{BBD}).
In the next result we obtain an algebraic description of the kernel of $\phi_{(\epsilon_1,\, \epsilon_2,\, \epsilon_3)}\colon VSG_2\to S_2$.
In particular, the virtual singular pure braid group $VSPG_2$ is isomorphic to $FVP_3$. Recall that an HNN-extension is a group obtained by adjoining a stable letter that conjugates one subgroup onto another isomorphic subgroup.

\begin{cor}\label{cor:kernels}
Let $\epsilon_k\in \{0,1\}$, for $k=1,2,3$.
The kernel of $\phi_{(\epsilon_1,\, \epsilon_2,\, \epsilon_3)}\colon VSG_2\to S_2$ has the following properties, depending on the triple $(\epsilon_1,\, \epsilon_2,\, \epsilon_3)$:
\begin{itemize}

\item[(1,1,1):]
$\ker{\phi_{(1,1,1)}}=VSPG_2$ is an HNN-extension of the free group of rank two generated by $b_{1,2}, c_{1,2}$ with stable element $a_{1,2}$ and with associated subgroups $A=\ang{c_{1,2}b_{1,2}}$ and $B=\ang{b_{1,2} c_{1,2}}$, which are infinite cyclic.
Furthermore, $VSPG_2$ is isomorphic to the flat virtual pure braid group $FVP_3$ and, as a consequence, it is isomorphic to the free product $\Z^2\ast \Z$.

\item[(1,1,0):]
The group $\ker{\phi_{(1,1,0)}}$ is isomorphic to the free product $F_2\ast \Z_2\ast \Z_2$, where $F_2$ is a free group of rank $2$.

\item[(1,0,1):]
The kernel $\ker{\phi_{(1,0,1)}}=VST_2$ is isomorphic to the virtual singular pure braid group $VSPG_2$.

\item[(1,0,0):]
The group $\ker{\phi_{(1,0,0)}}$ is isomorphic to the free product $\Z^2\ast \Z_2\ast \Z_2$.

\item[(0,1,1):]
The group $\ker{\phi_{(0,1,1)}}$ is isomorphic to the virtual singular pure braid group $VSPG_2$.

\item[(0,1,0):]
The group $\ker{\phi_{(0,1,0)}}$ is isomorphic to the group $\ker{\phi_{(1,0,0)}}$.

\item[(0,0,1):]
The group $\ker{\phi_{(0,0,1)}}=VSK_2$ is a right-angled Artin group (RAAG) and it is isomorphic to the free product $\Z^2\ast \Z^2$.

\item[(0,0,0):]
The virtual singular braid group with two strands $VSG_2$ is isomorphic to $\Z^2\ast \Z_2$.

\end{itemize}
\end{cor}

\begin{proof}
This follows from a direct computation using the presentations in Theorem~\ref{thm:kernels}. 
\end{proof}

\begin{que}
It was proved in Corollary~\ref{cor:kernels} that $VST_2$ is isomorphic to $VSPG_2$.
It is natural to ask whether, for general $n$, one has $VST_n \cong VSPG_n$.
\end{que}

\begin{cor}
Let $\epsilon_k\in \{0,1\}$, for $k=1,2,3$.
For any triple $(\epsilon_1,\, \epsilon_2,\, \epsilon_3)$, the kernel of $\phi_{(\epsilon_1,\, \epsilon_2,\, \epsilon_3)}\colon VSG_2\to S_2$ has trivial center. In particular, the same holds for $VSG_2$.
\end{cor}

\begin{proof}
This follows immediately since free products of non-trivial groups have trivial center.
\end{proof}

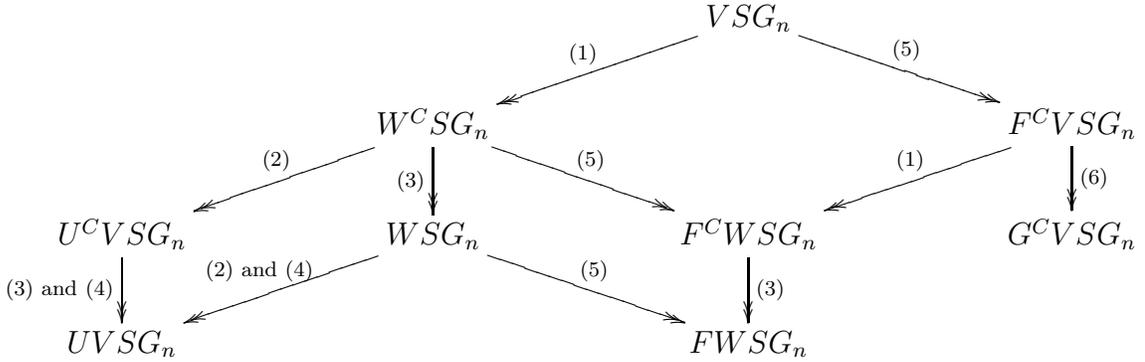
\section{Some quotients of virtual singular braid groups}\label{sec:quotients}

In this section we introduce and study several quotients of the virtual singular braid group $VSG_n$.
The motivation for considering such quotients comes from the rich theory of quotients of the classical and virtual braid groups, which have played an important role in knot theory, low-dimensional topology, and related areas (see, for instance, \cite{B,BBD,D,Kau,KL}).

A fundamental feature of virtual braid groups is the presence of so-called \emph{forbidden relations}, which do not hold in general but whose imposition leads to important and well-studied quotient groups, such as the welded braid group and the unrestricted virtual braid group.
In the virtual singular setting, similar phenomena occur. We first show that certain relations are forbidden in $VSG_n$.
Then, motivated by the virtual case, we introduce natural quotients of $VSG_n$ obtained by imposing some of these relations and study their algebraic properties.

\subsection{Forbidden relations in $VSG_n$}

We begin by identifying relations that do not hold in the virtual singular braid group.
These relations are natural analogues of the forbidden relations appearing in the theory of virtual braid groups and play a central role in the definition of several quotients of $VSG_n$.

\begin{thm}\label{thm:forbidden}
Let $n \geq 3$.
The following relations do not hold in the virtual singular braid group $VSG_n$, for all \(1 \leq i \leq n-2\):
\begin{align}
\sigma_i \sigma_{i+1} v_i &\neq v_{i+1} \sigma_i \sigma_{i+1}, \label{eq:forbidden1}\\
v_i \sigma_{i+1} \sigma_i &\neq \sigma_{i+1} \sigma_i v_{i+1}, \label{eq:forbidden2}\\
\tau_i \sigma_{i+1} v_i &\neq v_{i+1} \sigma_i \tau_{i+1}, \label{eq:forbidden3}\\
v_i \sigma_{i+1} \tau_i &\neq \tau_{i+1} \sigma_i v_{i+1}. \label{eq:forbidden4}
\end{align}
\end{thm}

\begin{proof}
We prove the result by showing that each of the relations above is not preserved under a suitable homomorphism from $VSG_n$ to a symmetric group.

Consider the homomorphism
\[
\phi_{(1,1,1)}\colon VSG_n \longrightarrow S_n
\]
defined in Definition~\ref{defn:phimap}, which sends each generator $\sigma_i$, $\tau_i$ and $v_i$ to the transposition $(i\ i+1)$.

Applying $\phi_{(1,1,1)}$ to the left-hand side and right-hand side of \eqref{eq:forbidden1}, we obtain
\[
\phi_{(1,1,1)}(\sigma_i \sigma_{i+1} v_i)
= (i\ i+1)(i+1\ i+2)(i\ i+1),
\]
and
\[
\phi_{(1,1,1)}(v_{i+1} \sigma_i \sigma_{i+1})
= (i+1\ i+2)(i\ i+1)(i+1\ i+2).
\]
Since these permutations are distinct in $S_n$, relation~\eqref{eq:forbidden1} does not hold in $VSG_n$.
The same argument applies to \eqref{eq:forbidden2}, \eqref{eq:forbidden3}, and \eqref{eq:forbidden4}, showing that none of these relations is valid in $VSG_n$.
\end{proof}

\begin{rem}
Relations of the form \eqref{eq:forbidden1} and \eqref{eq:forbidden2} are the classical forbidden relations in the theory of virtual braid groups.
Relations \eqref{eq:forbidden3} and \eqref{eq:forbidden4} arise naturally in the virtual singular setting and involve interactions between classical, singular, and virtual generators.
Imposing some or all of these relations leads to meaningful quotients of $VSG_n$, which we study in the following subsections.
\end{rem}

\subsection{The welded singular braid group}

Motivated by the theory of virtual braid groups, an important quotient of the virtual singular braid group is obtained by imposing one of the forbidden relations involving classical and virtual crossings.
In the virtual case, this leads to the welded braid group, which admits a rich geometric interpretation in terms of motions of unknotted circles in $\mathbb{R}^3$ (see, for instance, \cite{D,Kau}).
We introduce here the analogous construction in the virtual singular setting.

\begin{defn}\label{defn:wsb}
The \emph{welded singular braid group} on $n$ strands, denoted by $WSG_n$, is the quotient of the virtual singular braid group $VSG_n$ by the normal closure of the relations
\begin{equation}\label{eq:welded}
\sigma_i \sigma_{i+1} v_i = v_{i+1} \sigma_i \sigma_{i+1},
\qquad \text{for } 1 \le i \le n-2.
\end{equation}
\end{defn}

Equivalently, the group $WSG_n$ is obtained from $VSG_n$ by imposing the first forbidden relation \eqref{eq:forbidden1} of Theorem~\ref{thm:forbidden}.
As shown in that theorem, relation~\eqref{eq:welded} does not hold in $VSG_n$, so $WSG_n$ is a proper quotient of the virtual singular braid group.

\begin{rem}
Relation~\eqref{eq:welded} is usually referred to as the \emph{over-forbidden relation}.
In the virtual braid group setting, imposing this relation yields the welded braid group.
The definition above can therefore be seen as a natural singular analogue of that construction.
\end{rem}

The welded singular braid group retains many structural features of $VSG_n$.
In particular, the natural projection onto the symmetric group remains well defined.

\begin{prop}\label{prop:wsb-projection}
The canonical homomorphism
\[
\phi_{(1,1,1)}\colon VSG_n \longrightarrow S_n
\]
induces a surjective homomorphism
\[
\overline{\phi}_{1,1,1}\colon WSG_n \longrightarrow S_n.
\]
\end{prop}

\begin{proof}
The relation~\eqref{eq:welded} is preserved under the homomorphism $\phi_{(1,1,1)}$, since both sides are mapped to the same permutation in $S_n$.
Therefore, $\phi_{(1,1,1)}$ factors through the quotient defining $WSG_n$, yielding a well-defined surjective homomorphism $\overline{\phi}_{1,1,1}$.
\end{proof}

\begin{rem}
As in the case of $VSG_n$, the kernel of $\overline{\phi}_{1,1,1}$ may be regarded as a \emph{welded singular pure braid group}.
A detailed study of this subgroup is beyond the scope of the present paper, but many of the techniques developed in Section~\ref{vsgn} apply verbatim in this setting.
\end{rem}

\subsection{The unrestricted virtual singular braid group}

Another natural quotient of the virtual singular braid group is obtained by imposing \emph{both} forbidden relations involving classical and virtual generators.
In the virtual braid group setting, this leads to the unrestricted virtual braid group, which has been studied extensively (see, for instance, \cite{B,BBD,Kau,KL}).
We introduce here the corresponding construction in the virtual singular context.

\begin{defn}\label{defn:uvsg}
The \emph{unrestricted virtual singular braid group} on $n$ strands, denoted by $UVSG_n$, is the quotient of the virtual singular braid group $VSG_n$ by the normal closure of the relations
\begin{align}
\sigma_i \sigma_{i+1} v_i &= v_{i+1} \sigma_i \sigma_{i+1}, \label{eq:unrestricted1}\\
v_i \sigma_{i+1} \sigma_i &= \sigma_{i+1} \sigma_i v_{i+1}, \label{eq:unrestricted2}
\end{align}
for all $1 \le i \le n-2$.
\end{defn}

Equivalently, the group $UVSG_n$ is obtained from $VSG_n$ by imposing both classical forbidden relations \eqref{eq:forbidden1} and \eqref{eq:forbidden2} of Theorem~\ref{thm:forbidden}.
Since neither of these relations hold in $VSG_n$, the group $UVSG_n$ is a proper quotient of the virtual singular braid group.

\begin{rem}
In the virtual braid group case, imposing relations \eqref{eq:unrestricted1} and \eqref{eq:unrestricted2} yields the unrestricted virtual braid group.
Thus, Definition~\ref{defn:uvsg} provides a direct singular analogue of that construction.
\end{rem}

As in the case of $VSG_n$ and of the welded singular braid group, the natural projection onto the symmetric group is still available.

\begin{prop}\label{prop:uvsg-projection}
The canonical homomorphism
\[
\phi_{(1,1,1)}\colon VSG_n \longrightarrow S_n
\]
induces a surjective homomorphism
\[
\overline{\phi}_{1,1,1}\colon UVSG_n \longrightarrow S_n.
\]
\end{prop}

\begin{proof}
Both relations \eqref{eq:unrestricted1} and \eqref{eq:unrestricted2} are preserved under $\phi_{(1,1,1)}$, since the images of their left-hand and right-hand sides coincide in $S_n$.
Therefore, $\phi_{(1,1,1)}$ factors through the quotient defining $UVSG_n$, yielding a well-defined surjective homomorphism $\overline{\phi}_{1,1,1}$.
\end{proof}

\begin{rem}
The kernel of $\overline{\phi}_{1,1,1}$ may be regarded as an \emph{unrestricted virtual singular pure braid group}.
Although we do not pursue a detailed study of this subgroup here, many of the methods developed in Section~\ref{vsgn} apply equally well in this setting.
\end{rem}

\subsection{Comparison with the virtual singular braid group}

We conclude this section by comparing the virtual singular braid group $VSG_n$ with two of its most natural quotients introduced above, namely the welded singular braid group $WSG_n$ and the unrestricted virtual singular braid group $UVSG_n$.
Our aim is to clarify which structural properties established for $VSG_n$ persist under these quotients and to describe the algebraic effects of imposing forbidden relations.

Recall that none of the relations listed in Theorem~\ref{thm:forbidden} hold in $VSG_n$.
Imposing the first family of forbidden relations yields the welded singular braid group $WSG_n$, while imposing both families yields the unrestricted virtual singular braid group $UVSG_n$.
In particular, for $n\ge 3$, we obtain proper quotients
\[
VSG_n \twoheadrightarrow WSG_n \twoheadrightarrow UVSG_n.
\]

\medskip

Despite the introduction of these additional relations, several fundamental properties of $VSG_n$ remain valid for both $WSG_n$
and $UVSG_n$.
In particular, the abelian invariants defined in Section~\ref{vsgn} survive in these quotients and give rise to short exact sequences analogous to those obtained for $VSG_n$.

\begin{prop}\label{prop:goldb-welded}
Let $n\ge 2$ and let $X_n$ be either $WSG_n$ or $UVSG_n$.
There exists a short exact sequence
\[
1 \longrightarrow \ang{B_n}^{X_n}
\longrightarrow X_n
\longrightarrow \Z \times \Z_2
\longrightarrow 1,
\]
where $\ang{B_n}^{X_n}$ denotes the normal closure of the classical braid group $B_n$ inside $X_n$.
\end{prop}

\begin{proof}
The proof follows the same lines as that of Theorem~\ref{thm:goldbsingular}.
Indeed, the homomorphisms defining the corresponding abelian invariants of $VSG_n$ factor through the quotients $WSG_n$ and $UVSG_n$, since the additional defining relations are preserved.
\end{proof}

\medskip

As in the virtual singular braid group, the canonical projection onto the symmetric group survives in both quotients.

\begin{prop}\label{prop:comparison-projection}
Let $n\ge 2$ and let $X_n$ be either $WSG_n$ or $UVSG_n$.
The canonical homomorphism
\[
\phi^{X_n}_{1,1,1}\colon X_n \longrightarrow S_n
\]
is surjective and induces a semi-direct product decomposition
\[
X_n=\ker{\phi^{X_n}_{1,1,1}} \rtimes S_n.
\]
\end{prop}

\begin{proof}
The defining relations of $WSG_n$ and $UVSG_n$ are compatible with the homomorphism $\phi_{(1,1,1)}$, and therefore the argument is identical to that of Proposition~\ref{prop:semi-direct}.
\end{proof}

On the other hand, the imposition of forbidden relations has a visible impact on the algebraic complexity of these groups.
In $WSG_n$ and $UVSG_n$, certain mixed relations involving classical, singular and virtual generators become interchangeable, leading to simplifications that do not occur in $VSG_n$.
This phenomenon mirrors what happens in the virtual braid group setting and supports the interpretation of $WSG_n$ and $UVSG_n$ as natural singular analogues of the welded and unrestricted virtual braid groups.

Finally, we describe homomorphisms from these quotients to symmetric groups, in direct analogy with the classification obtained for $VSG_n$ in Theorem~\ref{thm:BP}. 
Let $\nu_6$ denote the exceptional outer automorphism of the symmetric group $S_6$.

\begin{prop}\label{thm:BP-welded}
Let $n,m$ such that $n\ge 5$, $m\ge 2$ and $n\ge m$.
Let $X_n$ be either $WSG_n$ or $UVSG_n$, and let $\psi^{X_n}\colon X_n \to S_m$ be a group homomorphism.
Then, up to conjugation, one of the following possibilities holds:
\begin{enumerate}
	\item $\psi^{X_n}$ is abelian;
	\item $n=m$ and $\psi^{X_n}=\phi^{X_n}_{1,1,1}$;
	\item $n=m=6$ and $\psi^{X_n}=\nu_6\circ\phi^{X_n}_{1,1,1}$.
\end{enumerate}
\end{prop}

\begin{proof}

The proof follows closely the argument of \rethm{BP}.
Indeed, the additional defining relations of $WSG_n$ and $UVSG_n$ are preserved under the homomorphisms involved and do not interfere with the factorization arguments or with the classification of homomorphisms into symmetric groups.
\end{proof}

\end{document}